\documentclass[letterpaper,journal,twoside,print,10pt]{ieeecolor}
\usepackage{generic}
\usepackage{verbatim}

\usepackage{etoolbox}
\makeatletter
\@ifundefined{color@begingroup}%
  {\let\color@begingroup\relax
   \let\color@endgroup\relax}{}%
\def\fix@ieeecolor@hbox#1{%
  \hbox{\color@begingroup#1\color@endgroup}}
\patchcmd\@makecaption{\hbox}{\fix@ieeecolor@hbox}{}{\FAILED}
\patchcmd\@makecaption{\hbox}{\fix@ieeecolor@hbox}{}{\FAILED}

\def\BibTeX{{\rm B\kern-.05em{\sc i\kern-.025em b}\kern-.08em
    T\kern-.1667em\lower.7ex\hbox{E}\kern-.125emX}}
\usepackage{graphicx,multirow,booktabs, subfig}
\usepackage{amsmath,amssymb,amsthm, bm}
\usepackage{cases}
\usepackage[T1]{fontenc}
\usepackage{hyperref}
\allowdisplaybreaks
\hypersetup{
     colorlinks =true,
     linkcolor = blue,
     filecolor = blue,
     citecolor = magenta,      
     urlcolor = magenta,
     }
\usepackage{setspace}
\usepackage{siunitx}
\usepackage{csvsimple}
\usepackage{xurl}
\usepackage[scaled=0.85]{helvet}
\usepackage[T1]{fontenc}

\usepackage{mathalpha}

\usepackage[style=ieee,doi=false]{biblatex}
\AtBeginBibliography{\small}
\allowdisplaybreaks
\let\labelindent\relax
\usepackage{enumitem, xspace}

\usepackage{accents}
\usepackage{xfrac}

\newcommand{\GFp}{\mathcal{G}\mathcal{F}_p}
\newcommand{\GFpi}{\mathcal{G}\mathcal{F}_{\pi}}

\renewcommand{\bf}[1]{\mathbf{#1}}
\newcommand{\bs}[1]{\boldsymbol{#1}}

\newcommand{\mR}{\mathbb{R}}

\theoremstyle{remark}
\newtheorem{remark}{Remark}
\newtheorem{definition}{Definition}

\newtheorem{proposition}{Proposition}
\newtheorem{lemma}{Lemma}
\newtheorem{theorem}{Theorem}
\newtheorem{corollary}{Corollary}

\newcommand{\review}[1]{\textcolor{black}{#1}}

\begin{document}

\title{On the Existence of Steady-State Solutions to the Equations Governing Fluid Flow in Networks}

\author{Shriram Srinivasan$^{\dagger}$, Nishant Panda$^{\ddagger}$, and Kaarthik Sundar$^{*}$
\thanks{$^{\dagger}$Applied Mathematics and Plasma Physics Group, Los Alamos National
Laboratory, Los Alamos, New Mexico, USA. E-mail: \texttt{shrirams@lanl.gov}}\;
\thanks{$^{\ddagger}$Information Sciences Group, Los Alamos National Laboratory, Los Alamos, New Mexico, USA. E-mail: \texttt{nishpan@lanl.gov}}
\thanks{$^{*}$Information Systems and Modeling Group, Los Alamos National
Laboratory, Los Alamos, New Mexico, USA. E-mail: \texttt{kaarthik@lanl.gov}}\;
\thanks{SS and KS acknowledge funding provided by LANL’s Directed Research and Development (LDRD) project 20220006ER. SS and NP  thank LDRD for support through 20230480ECR and 20230254ER respectively. The research work conducted at Los Alamos National Laboratory is done under the auspices of the National Nuclear Security Administration of the U.S. Department of Energy under Contract No. 89233218CNA000001. }\; 
}

\maketitle

\begin{abstract}
The steady-state solution of fluid flow in pipeline infrastructure networks driven by junction/node potentials is a crucial ingredient in various decision-support tools for system design and operation. While the \review{nonlinear} system is known to have a unique solution (when one exists), the absence of a definite result on the existence of solutions hobbles the development of computational algorithms, for it is not possible to distinguish between algorithm failure and non-existence of a solution.
In this letter, we show that for any fluid whose equation of state is a scaled monomial, a unique solution exists for such \review{nonlinear} systems if the term \emph{solution} is interpreted in terms of \emph{potentials} and flows rather than \emph{pressures} and flows. However, for gases following the CNGA equation of state, while the question of existence remains open, we construct an alternative system that always has a unique solution and show that the solution to this system is a good approximant of the true solution. The existence result for flow of natural gas in networks also applies to other fluid flow networks such as water distribution networks or networks that transport carbon dioxide in carbon capture and sequestration. 
Most importantly, our result enables correct diagnosis of algorithmic failure, problem stiffness, and non-convergence in computational algorithms.
\end{abstract}

\IEEEpeerreviewmaketitle

\section{Introduction} \label{sec:intro}

\IEEEPARstart{F}LUID flow in pipeline infrastructure networks driven by junction/node potentials through a network consisting of pipelines and pressure-regulating devices (like pumps/compressors) is relevant in the context of energy distribution and environmental remediation,  such as in the delivery of hydrogen/natural gas, water distribution, or transport of carbon dioxide in carbon capture and sequestration.  Under steady-state conditions,  the governing equations (which are coupled PDEs \cite{nrsolver}) reduce to a \review{nonlinear} system of algebraic equations that relate pressures at junctions/nodes with mass flows of the pipes and other devices throughout the network.  The solution of such a \review{nonlinear} system is a crucial ingredient in various decision support tools for system design and operation to ensure the safe and reliable transport of fluids through the network.

\review{Nonlinear} algebraic systems may have no solution, a unique solution, or even multiple solutions, and before designing algorithms to solve such systems, it is important to answer the question of existence and uniqueness of solutions if possible. The question of existence is not an abstruse mathematical curiosity but one that has fundamental significance from the viewpoint of modeling, control, and computation. The equations representing the model of a real system are a mathematical construct that aims to mimic the physical response of real systems. While the occurrence of a physical phenomenon may be obvious and undeniable, the verification that the equations can simulate this phenomenon with fidelity is tantamount to proving the existence of a solution under certain hypotheses or assumptions. Thus, investigation of existence of a solution is linked to the validity of the model itself.  
In the control of linear systems, the existence of a solution is rarely spoken about since most linear ODE systems under consideration are almost always regular and guarantee existence and uniqueness.  However, in the case of nonlinear systems, notions such as controllability cannot be well-defined unless it is known that a solution exists \cite{slotinebook}. Finally, from the perspective of computation, the existence of a solution is an assurance that if successively finer numerical approximations do not converge, it is the algorithm that needs modification and not the problem.

Notwithstanding the connection to these fundamental issues,  rarely do systems arising in practice afford us either the luxury of first proving/determining existence of a solution before turning to algorithms for approximating it, or the courtesy of readily satisfying the hypotheses required for the invocation of some mathematical theorem that can assert the existence of a solution.  Thus, research efforts march onwards to develop computational algorithms \emph{assuming} that a solution exists,  with uniqueness of solution lending support when applicable.  Such is the case even with a well-established, justly celebrated model like the Navier-Stokes equation for incompressible fluids.
The subject of this article is the steady-state equations governing adiabatic, unidirectional\footnote{\review{the only non-zero component of flow/velocity is along the length of the edge}} fluid flow in networks. The canonical form of these equations corresponds to different contexts, such as natural gas delivery, water distribution, or transport of carbon dioxide in carbon capture and sequestration. \review{However, the system of equations corresponding to natural gas subsume the other cases and thus the analysis and conclusions obtained herein for natural gas apply without caveat to the other fluid flow systems.} 

For (natural) gas flow, the \review{nonlinear} \review{system of equations} has been shown to have a unique solution, if one exists \cite{nrsolver,Singh2019,Singh2020,misra2020monotonicity}, and several attempts have been made to develop efficient computational algorithms \cite{Dvijotham2015Jun,ojha2017solving,Singh2020, nrsolver} agnostic to the grid topology that can determine a solution.  However, none of these algorithms succeed with large grid sizes, and the absence of a definite result on existence of solutions hobbles further development; for it is not possible to distinguish between algorithm failure and non-existence of a solution. \review{Moreover, non-existence of solutions (pressures and flows) has been shown \cite{nrsolver} even for the simplest case of flow in a single pipe.}

In this letter, we  show that \review{for any fluid whose equation of state is a scaled monomial, such as the ideal gas}, a unique solution exists for these \review{nonlinear}  gas flow systems if the term \emph{solution} is interpreted in terms of \emph{potentials} and flows rather than \emph{pressures} and flows \review{as in \cite{nrsolver}} and this result applies to other fluid flow networks as well.

We use \emph{topological degree theory} and \emph{homotopy invariance} \cite{orteganonlinbook,benevieri2023intro} to assert existence of a solution by starting with a linear system (where the degree computation is trivial) that can be continuously transformed to obtain the \review{nonlinear} system of interest. 
There are subtle but crucial differences from \emph{homotopic continuation} methods \cite{Watson1989May,allgower2003introduction} which are constructive and postulate a trajectory connecting the  solution of a known problem to the solution of the \review{nonlinear} system of interest. Homotopic continuation is mainly used as a numerical algorithm \cite{Lima-Silva2023Jul} and requires that the system satisfy very stringent hypotheses that are verifiable only in rare cases \cite{Chiang2017Feb} where the system has some special structure. In contrast, assertions of existence following degree theory and homotopy invariance as used here require weaker assumptions that even allow  the Jacobian to be singular at a solution.

The article is structured as follows: 
\review{We state and summarize the mathematical propositions and assertions required in Section~\ref{sec:background}.
After setting up the notation and recapitulating the governing equations in Section~\ref{sec:prelims},  we prove the main result in Section~\ref{sec:results} and conclude with Section~\ref{sec:conclusion}.}
\section{\review{Mathematical Background}}
\label{sec:background}
The following results from topological degree theory and homotopy invariance will be invoked subsequently. Proofs and other details of degree theory may be found in \cite[Chapter~6]{orteganonlinbook} or \cite{benevieri2023intro}.
\begin{definition}
 \label{def:regular}
A point $\bs y$ is a \emph{regular} value of a map $\bs{f}: \mathbb{R}^{n} \to \mathbb{R}^{n}$ if $\det(  \bs f'(\bs x) ) \neq 0 \; \forall \; \bs{x} \in \bs{f}^{-1} \{\bs y\} $.
\end{definition}
\begin{definition}
\label{def:degree}
Given a map $\bs{f}: \mathbb{R}^{n} \to \mathbb{R}^{n}$ and an open bounded domain  $\mathcal{D} \subset \mathbb{R}^{n}$ such that (1) $\bs{y}$ is a \emph{regular} value of $\bs{f}$, (2) $\bs{f}$ is continuously differentiable on $\mathcal{D}$, and (3) $\bs{y} \notin \bs{f}(\partial\mathcal{D})$, 
the \emph{degree} of $\bs{f}$ at $\bs{y}$ over $\mathcal{D}$ is defined as 
$$\mathrm{deg}(\bs{f}, \mathcal{D}, \bs{y}) = \begin{cases}  0 & \mathrm{if} \; \bs{f}^{-1} \{\boldsymbol{y}\} \; \textrm{is empty}. \\ 
 \sum\limits_{n}\mathrm{sgn}( \det( \bs f'(\bs x_n) ) ) & \forall \; \bs x_n \in \bs{f}^{-1} \{\bs{y}\}. 
\end{cases}$$
\end{definition}
\begin{remark}
\label{remark:degree-lin-op}
Given a domain $\mathcal{D} \subset \mR^{n}$ and an invertible linear operator $\bs A$, $\mathrm{deg}(\bs A, \mathcal{D}, \bs{y}) \neq 0$ for any   $\bs y \in \mathcal{D}$.
\end{remark}
\begin{proposition}
\label{thm:kronecker}
If $\mathrm{deg}(\bs{f}, \mathcal{D}, \bs{0}) \neq 0$, then  $\bs{f}$ has a zero/root in $\mathcal{D}$.
\end{proposition}
\begin{proposition}[Homotopy invariance]
\label{thm:homotopy}
Let $\mathcal{D} \subset \mathbb{R}^{n}$  be open and bounded and let $H: \mathrm{cl}({\mathcal{D}}) \times [0,1] \to \mathbb{R}^{n}$ be a continuous map on $\mathrm{cl}({\mathcal{D}}) \times [0,1]$ where $\mathrm{cl}({\mathcal{D}})$ is the \emph{closure} of the set $\mathcal{D}$. Suppose further that $H(\boldsymbol{x}, s) \neq \boldsymbol{0}$ in $\partial\mathcal{D} \times [0, 1]$. Then $\mathrm{deg}(H(\boldsymbol{x}, s), \mathcal{D}, \boldsymbol{0})$ is constant for all $s \in [0, 1]$. Thus if $\mathrm{deg}(H(\boldsymbol{x}, 0), \mathcal{D}, \boldsymbol{0}) \neq 0$, then  $H(\boldsymbol{x}, s) = \boldsymbol{0}$ has a solution in $\mathcal{D}$ for every $s \in [0, 1]$.
\end{proposition}
Note that the definition of degree can be extended to allow \emph{non-regular} values in Propositions~\ref{thm:kronecker}~and~\ref{thm:homotopy}.
\section{\review{Model Description}} \label{sec:prelims}
%
\review{The adiabatic, unidirectional flow of compressible gas in a single pipeline is described by the simplified Euler equations \cite{thorley87}. 
\begin{flalign}
  \frac{\partial\rho}{\partial t}+\frac{\partial \phi}{\partial x} = 0, \;
  \frac{\partial\phi}{\partial t} + \frac{\partial p}{\partial x}  = 
  -\lambda\frac{\phi |\phi|}{\rho(p)}, 		
  \label{eq:euler1}
\end{flalign}
where $\rho, p, \phi$ are the density, pressure, and mass flux, respectively.
The additional parameter $\lambda$ incorporates the friction factor and diameter of the pipe.  
For steady-state, the system simplifies to
\begin{flalign}
\phi'(x)= 0, \; \rho(p)p'(x)  = 
  -\lambda \cdot \phi |\phi|.
  \label{eq:ode}
\end{flalign}
The nature of the gas is prescribed by the equation of state $\rho(p)$ chosen to model the pressure-density relationship of the gas and it determines a potential function $\pi: \mR \to \mR$ through
\begin{equation}
\pi(p) = \int_{0}^{p} \rho(\tilde{p}) ~d\tilde{p},
\label{eq:pi-def}
\end{equation}
which allows \eqref{eq:ode} to be integrated to the algebraic form
$$ \phi_1 = \phi_2 = \phi_{12}, \; \pi(p_1) - \pi(p_2) = \lambda \cdot \phi_{12} |\phi_{12}|\cdot (x_2 - x_1).$$
}
Two important cases of interest are that of the ideal gas  and a non-ideal gas that follows the CNGA equation of state \cite{menon2005gas}  for which we have:
\begin{subequations}
\begin{gather}
   \rho(p) = p, \; \pi^{\mathrm{ideal}}(p) =  p^2/2\;,  \label{eq:pi-ideal}\\
   \rho(p) = b_1p + b_2 p^2, \; \pi^{\mathrm{cnga}}(p) = b_1p^2/2 + b_2p^3/3\;, \label{eq:pi-cnga} \\ 
   \text{with} ~~ b_1 = 1.003 ~\text{and}~ b_2 = 2.968\times10^{-2} \text{~\si{\per\mega\pascal}}. \label{eq:cnga-constants}
\end{gather}
\label{eq:Pi-expressions}
\end{subequations}
\review{In most cases, the equation of state $\rho(p)$ is a polynomial, and so the potential function also has a polynomial form.}

In order to describe the flow equations for a network of pipelines, we first introduce some notation. 
Let $G(V, E)$ denote the graph of the pipeline network, where $V$ and $E$ are the set of \emph{junctions} and \emph{edges} of the network, respectively. The edge set is given by $E = P \bigcup C$ where $P$ and $C$ are the set of pipes and compressors, respectively. A \emph{pipe} $(i, j)$ in $P$ connects the junctions $i$ and $j$ with a resistance coefficient $\beta_{ij} > 0$. Also, a \emph{compressor} $(i, j)$ in $C$ connects junctions $i$ and $j$, which we assume are geographically co-located.  For each junction $i$, we let $p_i, q_i$ denote the nodal value of the pressure and the injection at $i$, respectively.  For any edge  $(i, j)$ (pipe or compressor) that connects junctions $i$ and $j$, we let $\phi_{ij}$ denote the steady flow through the edge. Finally, for each compressor $(i, j)$, we use $\alpha_{ij} \geqslant 1$ to denote the specified \emph{pressure boost ratio or compressor ratio}. It is assumed that the direction of flow is in the direction of pressure boost, i.e., $i \to j$, and hence $\phi_{ij} > 0$ for a \emph{physically reasonable compressor} $(i, j)$. 
Each junction $i$  is assigned to one of two mutually disjoint sets consisting of the slack junctions $N_s$ and non-slack junctions $N_{ns}$ respectively, depending on whether pressure $p_i$ or gas injection $q_i$ at the node is specified. Thus $V = N_s \cup N_{ns}$.

The system of equations for the network  may now be summarized as follows \cite{nrsolver} ($*$ superscript indicates given data):
\begin{subnumcases} {\label{eq:GF-p}
\mathcal G \mathcal F_p:}  
\pi(p_i) - \pi(p_j) = \beta^{*}_{ij} \phi_{ij} | \phi_{ij}| \quad \forall (i, j) \in P, \label{eq:GF-p-pipe}\\ 
\alpha^{*}_{ij} p_i - p_j = 0 \quad \forall (i, j) \in C, \label{eq:GF-p-compressor} \\ 
\underset{(j, i) \in E }{\sum} \phi_{ji} - \underset{(i, j) \in E }{\sum} \phi_{ij} = q^{*}_i \; \forall i \in N_{ns}, \label{eq:GF-p-balance}\\
p_i = p^{*}_i \; \forall i \in N_s. \label{eq:GF-p-slack}
\end{subnumcases}
\review{
\begin{definition}
\label{def:p-sol}
Any $\bs{p} \in \mR^{|V|},\bs{\phi} \in \mR^{|E|}$ that satisfies 
the $\GFp$ system \eqref{eq:GF-p} is called a \emph{generalized pressure-solution}.
\end{definition}
The definition of \emph{generalized pressure-solution} allows the pressures and flows to take any real values. 
However, due to physical considerations we make a distinction between the equations being \emph{solvable} and \emph{feasible}.
As outlined in \cite{nrsolver}, for both the ideal and the non-ideal gas (with CNGA equation of state),  a \emph{feasible} solution will have non-negative pressures and non-negative flow in compressors. Otherwise, the solution/system is \emph{infeasible}.
}

There are certain assumptions we must make about the network in order to ensure the validity of the conclusions and assertions that follow. These assumptions are not restrictive, serving instead to eliminate certain impractical but pathological cases, and are stated as follows \cite{nrsolver}:
\begin{enumerate}[label=(A\arabic*)]
    \item There is at least one slack junction, i.e., $|N_s| \geqslant 1$.\label{assumption:slacks}
    \item When  $|N_s| \geqslant 2$, a path connecting two slack junctions must consist of at least one pipe. \label{assumption:path-pipe}
    \item Any cycle must consist of at least one pipe. \label{assumption:cycle-pipe}
\end{enumerate}
\begin{proposition}
\label{prop:uniqueness-GF-p}
Given assumptions (A1)--(A3), and a monotonic potential function $\pi$, a generalized-pressure solution to the $\GFp$ system \eqref{eq:GF-p}, if it exists, must be unique.
\end{proposition}
\begin{proof}
See \cite{Singh2019, nrsolver}.
\end{proof}
Anticipating subsequent development, we  note the following
\review{
\begin{lemma}
\label{lemma:pi_separability}
If $\pi, g$ are  real-valued, differentiable functions and for all $x, \alpha \in \mR$,  $\pi(\alpha x) = g(\alpha) \pi(x)$, then it is necessary and sufficient that $\pi$ be a \emph{scaled monomial}, i.e., $\pi(x) = k x^r$ for some  $k,  r \in \mR$. 
\end{lemma}
\begin{proof}
The sufficiency of the stated representation is obvious. For necessity, observe that if $x=0$, $\pi(0) = g(\alpha)\pi(0)$. If $\alpha =0$, $\pi(0) = g(0)\pi(x)$. These two conditions imply that either $\pi, g$ are constants or $\pi(0) = g(0) = 0$. Now consider  $\alpha x \neq 0$ for which $\pi(x) \neq 0, g(\alpha)\neq 0$.
Differentiating with respect to $x$ and $\alpha$, one has $\pi'(\alpha x) \alpha = g(\alpha) \pi'(x)$ and $\pi'(\alpha x) x = g'(\alpha) \pi(x)$ respectively. Together, they imply $\alpha g'(\alpha)/g(\alpha) = x \pi'(x)/\pi(x) = r$ (for some $r \in \mR$) from which we may conclude that $\pi(x) = k x^r$ for some $k \in \mR$.
\end{proof}
}
%
%
If we denote a nodal potential $\pi_i = \pi(p_i)\; \forall i \in V$, we can eliminate the pressures in favour of the potentials in \eqref{eq:GF-p-pipe} and \eqref{eq:GF-p-slack}. 
The equation~\eqref{eq:GF-p-compressor} relating pressures $p_i, p_j$  implies that $\pi(\alpha^{*}_{ij} p_i) =\pi_j$. 
\begin{remark}
When the potential function $\pi$ is a scaled monomial as in Lemma~\ref{lemma:pi_separability},  we obtain an equation relating potentials $\pi_i, \pi_j$ for some suitable $\gamma^{*}_{ij} = g(\alpha^{*}_{ij})$ and the system \eqref{eq:GF-p} is transformed into a new system:
\end{remark}
\begin{subnumcases} {\label{eq:GF-pi}
\mathcal G \mathcal F_{\pi}:}  
\pi_i - \pi_j = \beta^{*}_{ij} \phi_{ij} | \phi_{ij}| \quad \forall (i, j) \in P , &  \label{eq:GF-pi-pipe}\\ 
\gamma^{*}_{ij} \pi_i - \pi_j = 0 \quad \forall (i, j) \in C, & \label{eq:GF-pi-compressor} \\ 
\underset{(j, i) \in E }{\sum} \phi_{ji} - \underset{(i, j) \in E }{\sum} \phi_{ij} = q^{*}_i \; \forall i \in N_{ns}, & \label{eq:GF-pi-balance}\\
\pi_i = \pi^{*}_i \; \forall i \in N_s. & \label{eq:GF-pi-slack}
\end{subnumcases}

\begin{remark} 
\label{rem:approx}
The potential function $\pi = \pi^{\mathrm{ideal}}$ satisfies Lemma~\ref{lemma:pi_separability} (with $k=1/2, r=2$) while $\pi = \pi^{\mathrm{cnga}}$ does not.
Thus $\GFp$ system~\eqref{eq:GF-p} transforms to $\GFpi$ system~\eqref{eq:GF-pi} for the ideal gas equation of state but not for the CNGA equation of state. This will be investigated further in Sec.~\ref{sec:results}.
\end{remark}

\review{
\begin{definition}
\label{def:pi-sol}
A $\bs{\pi} \in \mR^{|V|},\bs{\phi} \in \mR^{|E|}$ that satisfies 
the $\GFpi$ system~\eqref{eq:GF-pi} is called a \emph{generalized potential-solution}.
\end{definition}
\begin{lemma}
\label{lemma:Pi-unique}
Given assumptions (A1)--(A3), a generalized-potential solution to $\GFpi$ system~\eqref{eq:GF-pi}, if it exists, must be unique.
\end{lemma}
\begin{proof}
One can interpret the $\GFpi$ system~\eqref{eq:GF-pi} as a special case of $\GFp$ system~\eqref{eq:GF-p} with the equation of state $\rho(p) = 1$, and  $\pi(p) = p$.
By Proposition~\ref{prop:uniqueness-GF-p}, solutions to the $\GFp$ system~\eqref{eq:GF-p}, if they exist, are unique. 
\end{proof}
%
For given Definitions~\ref{def:p-sol} and \ref{def:pi-sol} assuming Lemma~\ref{lemma:pi_separability}, note that solvability of $\GFp$ system~\eqref{eq:GF-p} implies solvability of $\GFpi$ system~\eqref{eq:GF-pi}. However, the converse is false, because  for some $\pi_i \in \mR$, there may be no $p_i \in \mR$ such that $\pi(p_i) =\pi_i$. 
}
%
%

The questions we would like to answer are the following:
\begin{enumerate}[label=(Q\arabic*)]
\item Do the systems~\eqref{eq:GF-p} or \eqref{eq:GF-pi} for an ideal gas have a solution?  
\item For a non-ideal gas (with the CNGA equation of state), does the system~\eqref{eq:GF-p} have a solution?
\end{enumerate}
For now, we can only prescribe a domain within which solutions for the system~\eqref{eq:GF-pi} (if they exist) must lie. This we do by estimating upper bounds using the given data as follows:
\begin{proposition}
 \label{prop:domain}
 Given a network $G(V, E)$, with sets $P, C$, $N_{ns}, N_s$ and data  as in system~\eqref{eq:GF-pi}, we define the following quantities:
 \begin{gather*}
 \beta_{m} \triangleq \max_{(i, j) \in P} \beta^*_{ij}, \; \alpha_{m} \triangleq \max_{(i, j) \in C} \alpha^*_{ij},  \; \phi_{m} \triangleq \max_{i \in N_{ns}} |q^*_i| \cdot |N_{ns}|, \\
 \pi_{m} \triangleq \left( \max_{i \in N_{s}} |\pi^*_i| +  |P|\beta_{m} \phi_{m}^2 \right) \cdot (\alpha_{m})^{|C|}.
 \end{gather*}
 Then the hypercube 
 \begin{equation}\label{eq:hypercube}
     \mathcal{D}_{\mathcal{H}} =\{\boldsymbol{x} \in \mR^{|V| + |E|}: |x_i| < \max(\phi_m, \pi_m)  \}
 \end{equation}
 is an open, bounded set that contains all solutions (if they exist) of system~\eqref{eq:GF-pi}.
\end{proposition}
These bounds have been derived by considering the maximum possible flow into any single pipe, then calculating the maximum possible pressure by considering  a node connected to the slack node with maximum slack pressure, and furthermore, taking the maximal possible compression achievable in the full network.

In Section~\ref{sec:results}, we shall find that the answer to (Q1) is that the system for an ideal gas has a unique generalized potential solution and if  the solution admits non-negative potentials, then it has a generalized pressure solution as well. 
\section{Results and Implications} 
\label{sec:results}
The stated facts in Section~\ref{sec:background} will now be put to use.
Before that however, we record the following result (restated with slight modification from \cite{nrsolver}) that guarantees the invertibility of a linear system with particular structure and is essential for the subsequent analysis.
\begin{proposition}
\label{prop:lin-system}
The  only solution to the following homogeneous linear system in $(\{x_k\}, \{y_{ij}\})$, where $k \in V$, $(i,j) \in E$ for an admissible gas network is $x_k = y_{ij} = 0$, where the real numbers  $a^*_i \neq 0$, all $a^*_i$ have the same sign, and for every cycle $\mathfrak{C} \in G(V, E)$, there exists at least one $(i,j) \in \mathfrak{C}$ such that $b^*_{ij} \neq 0$:
\begin{subequations}
\begin{align}
& a^*_i x_i - a^*_j x_j - b^*_{ij} y_{ij} = 0 \quad \forall (i, j) \in E, \\ 
& \underset{(j, i) \in E }{\sum} y_{ji} - \underset{(i, j) \in E }{\sum} y_{ij} = 0 \; \forall i \in N_{ns}, \\ 
& x_i = 0 \quad  \forall i \; \in N_s.
\end{align}
\label{eq:Jacobian_injective}
\end{subequations}
\end{proposition}
In order to invoke Proposition~\ref{thm:homotopy}, we interpret the solution of system ~\eqref{eq:GF-pi} as a zero or root of a \review{nonlinear} operator on $\mR^{|E| +|V|}$ and modify system~\eqref{eq:GF-pi} to introduce a scalar parameter $s \in [0, 1]$  such that 
the operator now \eqref{eq:GF-pi-s} has the form $$F: \mR^{|E|+|V|} \times \mR \longrightarrow \mR^{|E| + |V|}.$$ 
Setting $s=1$ in system~\eqref{eq:GF-pi-s}  recovers system~\eqref{eq:GF-pi}.
\begin{subnumcases} {\label{eq:GF-pi-s}
F((\bs{\phi}, \bs{\pi}), s):}  
\pi_i - \pi_j = \beta^{*}_{ij} \phi_{ij} | \phi_{ij}|^s \quad \forall (i, j) \in P , &  \label{eq:GF-pi-s-pipe}\\ 
\gamma^{*}_{ij} \pi_i - \pi_j = 0 \quad \forall (i, j) \in C, & \label{eq:GF-pi-s-compressor} \\ 
\underset{(j, i) \in E }{\sum} \phi_{ji} - \underset{(i, j) \in E }{\sum} \phi_{ij} = q^{*}_i \; \forall i \in N_{ns}, & \label{eq:GF-pi-s-balance}\\
\pi_i = \pi^{*}_i \; \forall i \in N_s. & \label{eq:GF-pi-s-slack}
\end{subnumcases}
\begin{remark}
$F((\bs{\phi}, \bs{\pi}), s)$ in \eqref{eq:GF-pi-s} is once-differentiable with a continuous derivative. 
\end{remark}

\begin{theorem}
\label{thm:existence-gfpi}
A  solution exists for $\GFpi$ system~\eqref{eq:GF-pi} and it is unique.
\end{theorem}

\begin{proof}
Observe that $F((\bs{\phi}, \bs{\pi}), 0)$ is a linear system of the same form as Proposition~\ref{prop:lin-system} and hence has a unique solution in the open hypercube $\mathcal{D}_{\mathcal{H}}$ in~\eqref{eq:hypercube} constructed in Proposition~\ref{prop:domain}. Moreover,  $F((\bs{\phi}, \bs{\pi}), s)$ is continuous on $\mathrm{cl}(\mathcal{D}_{\mathcal{H}})$ and by construction no solution can lie on the boundary of the hypercube. Hence by Proposition~\ref{thm:homotopy}, $F((\bs{\phi}, \bs{\pi}), s)$ has a solution in $\mathcal{D}_{\mathcal{H}}$ for every $s \in [0, 1]$. For $s=1$, the  system~\eqref{eq:GF-pi} then has a generalized-potential solution. Uniqueness follows from Lemma~\ref{lemma:Pi-unique}.
\end{proof}
The following corollary follows easily from Theorem~\ref{thm:existence-gfpi}:
\begin{corollary}
\label{corollary:invert-pi}
If the (unique) generalized-potential solution to \eqref{eq:GF-pi} has values of the potential $\pi_i \in \mR \; \forall i \in V$ such that they can be inverted to find pressures $p_i\in \mR$ satisfying $\pi(p_i) = \pi_i$, then a generalized-pressure solution exists for system~\eqref{eq:GF-p}. 
\end{corollary}
We note in passing that while inversion is possible only for non-negative values of $\pi^{\mathrm{ideal}}$ \eqref{eq:pi-ideal}, for $\pi^{\mathrm{cnga}}$ \eqref{eq:pi-cnga} a preimage can be found corresponding to any real value of the potential as the expression is a cubic polynomial.
\subsection{Ideal Gas Equation of State}
Theorem~\ref{thm:existence-gfpi} and Corollary~\ref{corollary:invert-pi}  have far-reaching implications, laying to rest all ambiguity about the solvability of the steady-state gas flow equations on networks for the case of the \emph{ideal gas}. 
The steady-state system for an ideal gas (corresponding to system~\eqref{eq:GF-pi})  \emph{always} has a unique, generalized potential solution. Furthermore, a generalized-pressure solution can be derived contingent on non-negative potentials.
In light of this statement, one may qualify the question (Q1) asked earlier to now read:
\begin{enumerate}[label=(Q\arabic*), start=3]
\item Do the systems~\eqref{eq:GF-p} or \eqref{eq:GF-pi} for an ideal gas have a \emph{feasible} solution?  
\end{enumerate}
From the theory developed thus far, one may definitively conclude apropos of  (Q3) that  the system~\eqref{eq:GF-pi} for the ideal gas is infeasible if any component of the generalized potential solution or compressor flow is negative, but feasible otherwise.

\subsection{Non-ideal Gas with CNGA Equation of State}
 \begin{figure*}[!htb]
\centering
\subfloat[$g(\alpha) = \gamma_{\mathrm{opt}}$]{\includegraphics[scale=0.65]{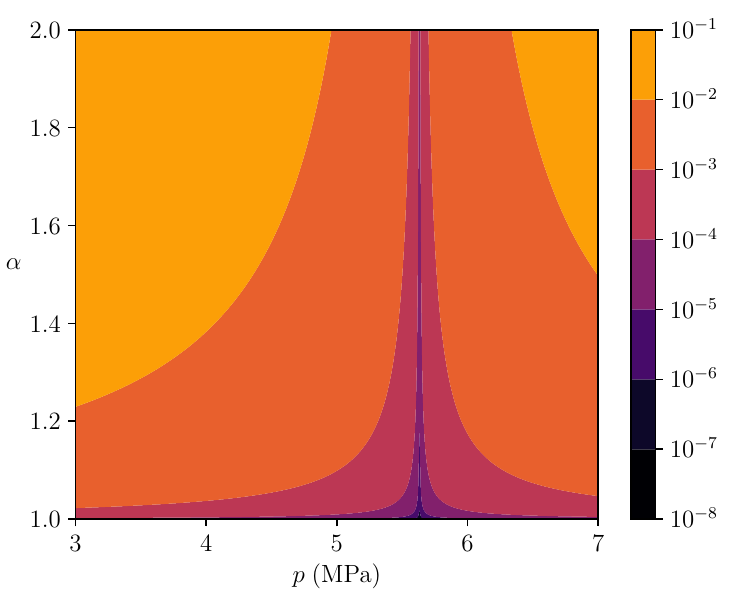}}\hfill
\subfloat[$g(\alpha) = \alpha^2$]{\includegraphics[scale=0.65]{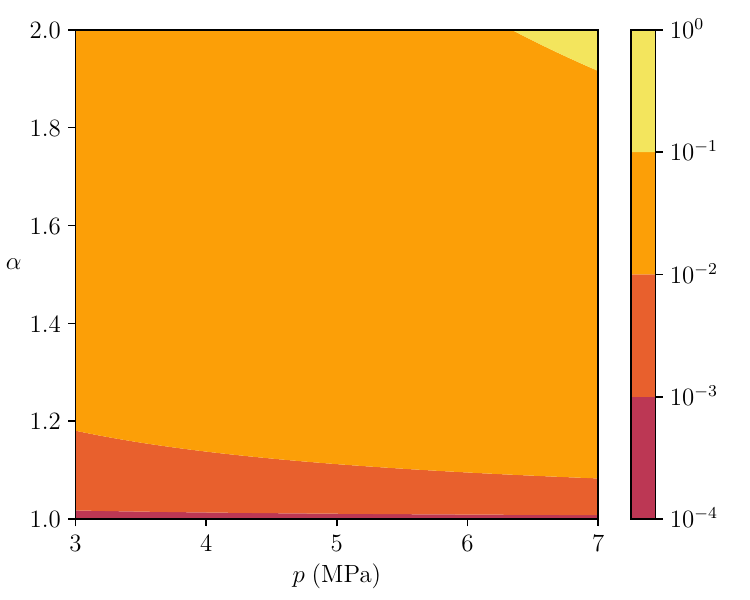}}
\caption{The plots above show the magnitude and distribution of the relative error $\mathrm{e}_{\mathrm{rel}} (\alpha, p)$ in \eqref{eq:rel_error} for two different choices of $g(\alpha)$ for the CNGA equation of state. The colour bars demonstrate how the error is reduced drastically by choosing the optimal $g(\alpha)$ in \eqref{eq:gamma-opt}. With this choice, we may write $\pi(\alpha p) \approx \gamma_{\mathrm{opt}}\pi(p)$ so that the system~\eqref{eq:GF-p} for the CNGA equation of state is approximated by \eqref{eq:GF-pi}.}
\label{fig:comp-approx}
\end{figure*}
We have thus far made no attempt to answer (Q2) articulated earlier because we already know generalized-pressure solutions may not exist \cite{nrsolver}. 
However, it is still of interest in the context of $\GFp$ system~\eqref{eq:GF-p}, to construct a continuous homotopy function $\chi(p, s)$  that satisfies $\chi(p, 0) = \pi^{\mathrm{ideal}}(p)$ and $\chi(p, 1) = \pi^{\mathrm{cnga}}(p)$ to replace $\pi(p)$ in system~\eqref{eq:GF-p} and examine why  Proposition~\ref{thm:homotopy} fails to apply. The failure is owing to  the term  $\chi(p, s)$, where it is not possible to estimate an upper bound for the solution analogous to Proposition~\ref{prop:domain}. 
Thus we are no closer to any kind of answer regarding existence of solution for the CNGA equation of state in (Q2).
However, it is natural to consider if the theory developed could prove useful in some other way.

Even though \eqref{eq:GF-p-compressor} \emph{does not imply} \eqref{eq:GF-pi-compressor}  for $\pi^{\mathrm{cnga}}$ \eqref{eq:pi-cnga}, can we \emph{approximate} $\pi(\alpha p) \approx g(\alpha) \pi(p)$ for some function $g$, so that $\GFpi$ \eqref{eq:GF-pi} approximates $\GFp$ \eqref{eq:GF-p} for the CNGA equation of state ? This approximation will be sensible only if the solution to    $\GFpi$ \eqref{eq:GF-pi} (which always exists by Theorem~\ref{thm:existence-gfpi}) is an acceptable approximant to the solution of  $\GFp$ \eqref{eq:GF-p} (when the latter exists). 
To investigate if that is so, we look to determine
\begin{equation*}
g(\alpha) \in \mathrm{span}\{1, \alpha,\alpha^2, \alpha^3\}
\end{equation*}
for the CNGA equation of state such that the residual $\pi(\alpha p) - g(\alpha)\pi(p)$ is minimized in a least-squares sense implied by the quantity 
\begin{equation*}
\int_{\alpha_{\mathrm{min}}}^{\alpha_{\mathrm{max}}}\int_{p_{\mathrm{min}}}^{p_{\mathrm{max}}} \left( \pi(\alpha p) - g(\alpha)\pi(p) \right)^2 \mathrm{d}p \; \mathrm{d}\alpha.
\end{equation*}
For an operating range of \numrange{1.0}{2.0} for the compressor ratio $\alpha$, with pressures within \SIrange{3}{7}{\mega\pascal}, we obtain 
\begin{equation}
\gamma_{\mathrm{opt}} = g(\alpha) \approx 0.9\alpha^2 +0.1\alpha^3.
\label{eq:gamma-opt}
\end{equation}
The significance of this choice can be evaluated by comparing the magnitude and the distribution of the relative error 
\begin{equation}
    \mathrm{e}_{\mathrm{rel}}(\alpha, p) = \left|\dfrac{\pi(\alpha p) - g(\alpha)\pi(p)}{\pi(\alpha p)} \right|
    \label{eq:rel_error}
\end{equation}
for $g(\alpha) = \gamma_{\mathrm{opt}}$ with that of the choice $g(\alpha) = \alpha^2$  in Figure~\ref{fig:comp-approx}. 
It can be seen that the error is reduced drastically by choosing the optimal $g(\alpha)$ in \eqref{eq:gamma-opt}. With this choice, we may write $\pi(\alpha p) \approx \gamma_{\mathrm{opt}}\pi(p)$ so that the system $\GFp$~\eqref{eq:GF-p} for the CNGA equation of state is approximated by $\GFpi$~\eqref{eq:GF-pi}.

Let $\bm x_a$ be the pressure-flux pair obtained from inversion of the potentials in the solution to $\GFpi$ system~\eqref{eq:GF-pi} with  the choice $\gamma_{\mathrm{opt}}$ in \eqref{eq:GF-pi-compressor}. Note that $\bm x_a$ can always be found since a unique generalized potential solution to $\GFpi$ system~\eqref{eq:GF-pi} always exists, and a pressure corresponding to a potential can always be found for the cubic polynomial \eqref{eq:pi-cnga}. 
We now compute a solution $\bm x$ (assuming it exists) to the original $\GFp$ system~\eqref{eq:GF-p}. If we denote by $\bm R(\bm y)$ the vector of residuals of the $\GFp$ system~\eqref{eq:GF-p} evaluated at $\bm y$, then $||\bm R(\bm x)||_{\infty} \approx 0$.

Motivated by standard ideas in the inversion of linear systems \cite{hornjohnson}, an approximate solution should have a small residual and also be close to the true solution. Thus, the two quantities that measure the acceptability of $\bm x_a$ as an approximation of $\bm x$ are the relative error and the residual expressed by
\begin{flalign}
\dfrac{||\bm x_a -\bm x||_{\infty}}{||\bm x||_{\infty}} \text{ and } ||\bm R(\bm x_a)||_{\infty}.
\label{eq:approx-measure}
\end{flalign}
The only qualitative assertion we can make is that the solution $\bm x_a$ ought to satisfy \eqref{eq:GF-p-pipe}, \eqref{eq:GF-p-slack} and that the  quantity $||\bm R(\bm x_a)||_{\infty}$ will be determined by  \eqref{eq:GF-p-compressor} alone. We also expect that a good approximate solution will serve as a good initial guess in a Newton-Raphson algorithm and lead to convergence within a handful of iterations.          

In order to support our claim, we perform numerical experiments on the GasLib-40 pipeline network \cite{Schmidt2017}.
The values for the pipe diameter, friction factor, nodal injections, etc., can be obtained at \url{https://gaslib.zib.de/} and \url{https://github.com/kaarthiksundar/ExistenceGasFlowRuns}. 
The junction with the largest injection in the network is chosen to be the slack junction with slack pressure set to \SI{5}{\mega\pascal}, and compressor ratios are set to a random value in the range \numrange{1.0}{2.0} for each compressor. We solve 1000 instances of the problem to obtain the approximate solution $\bm x_a$ and the true solution $\bm x$.
The resultant cumulative distribution obtained for the residual and the relative error \eqref{eq:approx-measure}  is shown in Figure~\ref{fig:cdf}.

These numerical experiments support our assertion that the solution $\bm x_a$ of system~\eqref{eq:GF-pi} is a good approximation to the computed solution $\bm x$ of system~\eqref{eq:GF-p} whenever the latter exists. Moreover, starting with the initial guess $\bm x_a$, the convergence to $\bm x$ was attained in exactly 2 Newton iterations in each instance.
The comparison between the solutions for an ideal gas and a non-ideal gas has already been analyzed in \cite{nrsolver} and is hence not considered here.
\begin{figure}
    \centering
    \includegraphics[scale=0.7]{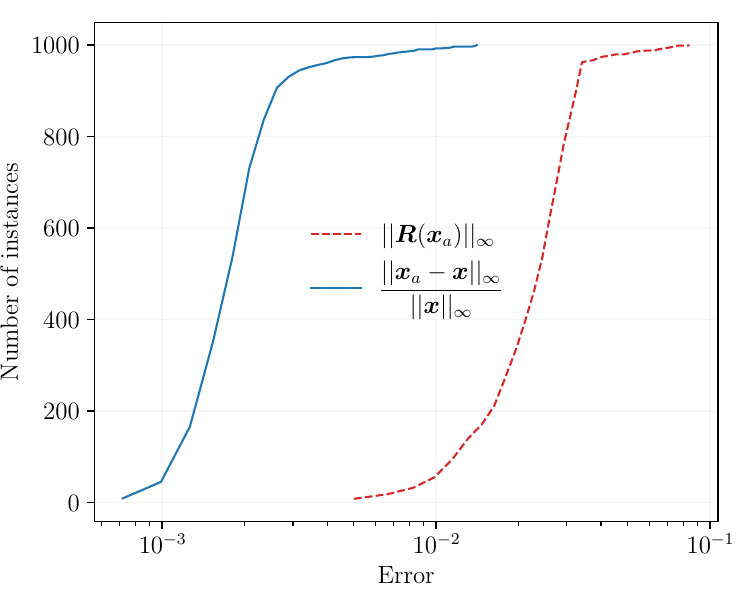}
    \caption{Cumulative distribution function for GasLib-40 showing the residual  $O(10^{-2})$ (red) as well as the relative error $O(10^{-3})$ (blue) defined by \eqref{eq:approx-measure} for the non-ideal gas. }
    \label{fig:cdf}
\end{figure}

%

\subsection{Generalizations}
The results obtained so far can be generalized to flow networks with different fluids. The three points of difference can come from the equation of state describing the fluid,  the model of the pumping device, and the viscous resistance function for the fluid (in this letter, the resistance function was $\phi|\phi|$).

In general, for any equation of state that has the form of a scaled monomial, the existence result in Theorem~\ref{thm:existence-gfpi} remains unchanged.
This is so in water distribution networks or in carbon capture and sequestration where the incompressibility of water and carbon dioxide implies that the equation of state is $\rho(p) = \rho^*$ (a constant) \cite{singh2019optimal}. 
The difference in the model of pumping devices could manifest in a given additive increase in potential or pressure \cite{tasseff2022polyhedral} in place of a given multiplicative increase considered in \eqref{eq:GF-p-compressor}  or \eqref{eq:GF-pi-compressor}, but Theorem~\ref{thm:existence-gfpi} still holds.
\review{Finally, Theorem~\ref{thm:existence-gfpi} continues to hold as long as the viscous resistance function chosen is monotonic and differentiable.}

\subsection{Use of Existence Result in Optimization}
\review{We shall now illustrate how the existence of a unique solution to $\GFpi$ can be used to detect infeasibilities in
the so-called optimal gas flow problem (OGF) that is solved regularly by natural gas pipeline network operators.} 

\review{The steady-state OGF \cite{Singh2019} is a nonlinear optimization problem described as follows:  Given  (vectors of) non-slack base nodal injections $\tilde{\bm q}$, slack nodal pressures $\tilde{\bm p}$, the minimum pressure value $p^{\min}$ at any non-slack node, and bounds ($\alpha^{\min}$ and $\alpha^{\max}$) for the compressor ratios, determine the vector of compressor ratios $\bm \alpha$ to \emph{minimize} the total compressor power \cite{menon2005gas} while satisfying the steady-state gas flow physics $\GFp$ Eq. \eqref{eq:GF-p} with $\bm q^* = \tilde{w}\cdot \tilde{\bm q}$, $\bm p^* =\tilde{\bm p}$, and $\bm \alpha^* =\bm \alpha$. Here a given injection scaling factor $\tilde{w} \geqslant 1$ is used to parametrize the OGF. }

\noindent \review{Mathematically, the OGF is formulated as: 
\begin{subnumcases} {\label{eq:ogf}
\mathrm{OGF}(\tilde w):}  
\min \sum_{(i, j) \in C} \phi_{ij}(\alpha_{ij}^m-1) \text{ subject to: } \label{eq:ogf-obj}\\ 
~~\mathcal G \mathcal F_p \text{ with ($\bm \alpha$, $\tilde{\bm p}$, $\tilde{w} \cdot \tilde{\bm q}$), } \label{eq:ogf-physics} \\ 
~~\alpha^{\min} \leqslant \alpha_{ij} \leqslant \alpha^{\max} ~\forall (i, j) \in C, \label{eq:ogf-alpha-limits}\\
~~\tilde{p}_i \geqslant p^{\min} ~\forall i \in N_{ns}. \label{eq:ogf-p-limits}
\end{subnumcases}
%
Eventually, $\tilde{\bm q}$ scaled by a large enough $\tilde{w}$ will cause withdrawals at some nodes large enough to lower nodal pressures below $p^{\min}$ even if compression ratios are at their maximum.}
\noindent \review{Thus, if we adopt  the following procedure -- (i) Set $\alpha_{ij}$ = $\alpha^{\max}$,  (ii) Solve $\GFpi$ (By Theorem~\ref{thm:existence-gfpi}, unique solution exists), we can guarantee that the $\mathrm{OGF}(\tilde w)$ is infeasible
if for any $i \in N_{ns}$, either $\pi_i < 0$ or  $p_i < p^{\min}$. 
For GasLib-40, this procedure shows that  $\mathrm{OGF}(\tilde w)$ is \emph{infeasible} for any $\tilde w \geqslant 2.135$ given $\alpha^{\min}=1$, $\alpha^{\max}=1.5$, and $p^{\min} =  300~ \si{psi}$. }

\review{There are no general techniques to provide infeasibility guarantees for nonlinear optimization problems and the infeasibility guarantees provided in this section was made possible solely due to the existence result for $\GFpi$ derived in Theorem~\ref{thm:existence-gfpi}. }

\section{Conclusion}
\label{sec:conclusion}

The arguments and proofs presented in this letter lay to rest the question of existence of solutions to potential-driven, steady-state fluid flow in a network when the potential functions are scaled monomials (Lemma~\ref{lemma:pi_separability}) such as for the ideal gas.
However, no conclusion can be drawn regarding  existence of solutions for potential functions that do not obey Lemma~\ref{lemma:pi_separability}. However, for gases following the CNGA equation of state, while the question of existence remains open, we construct an alternative system that always has a unique solution and show that the solution to this system is a good approximant of the true solution.
The results are also applicable to other networks carrying fluids such as water \cite{singh2019optimal} or carbon dioxide. Importantly, it gives support to computational work by enabling correct diagnosis of algorithmic failure, problem stiffness and non-convergence. Future work could look towards determining solutions by homotopic continuation methods \cite{allgower2003introduction} starting with a less complicated system with a guaranteed solution as exemplified by \cite{Chiang2017Feb}.

\printbibliography

\end{document}